\documentclass{amsart}
\usepackage{amsmath,amsthm,amssymb}
\usepackage[foot]{amsaddr}
\usepackage{comment}
\usepackage[numbers,sort]{natbib}
\usepackage{hyperref}
\usepackage{cleveref}
\usepackage[utf8]{inputenc}
\usepackage{todonotes}

\usepackage{xcolor}
\usepackage{enumerate}

\usepackage{bbold} 

\DeclareMathOperator{\supp}{supp}
\DeclareMathOperator{\diag}{diag}

\newcommand{\cG}{\mathcal{G}}
\newcommand{\cI}{\mathcal{I}}
\newcommand{\cB}{\mathcal{B}}
\newcommand{\cV}{\mathcal{V}}
\newcommand{\cE}{\mathcal{E}}
\newcommand{\cF}{\mathcal{F}}
\newcommand{\cX}{\mathcal{X}}
\newcommand{\cA}{\mathcal{A}}
\newcommand{\cS}{\mathcal{S}}
\newcommand{\cM}{\mathcal{M}}
\newcommand{\cL}{\mathcal{L}}
\newcommand{\cY}{\mathcal{Y}}

\newcommand{\real}{\mathbb{R}}
\newcommand{\complex}{\mathbb{C}}
\newcommand{\ifun}{\mathbb{1}} 
\newcommand{\eps}{\varepsilon}

\renewcommand{\Re}{\textrm{Re}}

\newtheorem{theorem}{Theorem}
\newtheorem{remark}{Remark}
\newtheorem{problem}{Problem}
\newtheorem{proposition}{Proposition}
\theoremstyle{definition}
\newtheorem{example}{Example}

\begin{document}
\title[Discrete inverse problems with internal functionals]{Discrete inverse problems with\\ internal functionals}

\author[M. Corbett]{Marcus Corbett}
\author[F. Guevara Vasquez]{Fernando Guevara Vasquez}
\email{fguevara@math.utah.edu}
\author[A. Royzman]{Alexander Royzman}
\author[G. Yang]{Guang Yang}
\address{Mathematics Department, University of Utah, Salt Lake City, UT 84112}

\thanks{This work was partially supported by the National Science Foundation grants DMS-2008610 and DMS-2136198.}



\begin{abstract}
    We study the problem of finding the resistors in a resistor network from measurements of the power dissipated by the resistors under different loads. We give sufficient conditions for local uniqueness, i.e. conditions that guarantee that the linearization of this non-linear inverse problem admits a unique solution. Our method is inspired by a method to study local uniqueness of inverse problems with internal functionals in the continuum, where the inverse problem is reformulated as a redundant system of differential equations. We use our method to derive local uniqueness conditions for other discrete inverse problems with internal functionals including a discrete analogue of the inverse Schr\"odinger problem and problems where the resistors are replaced by impedances and dissipated power at the zero and a positive frequency are available. Moreover, we show that the dissipated power measurements can be obtained from measurements of thermal noise induced currents. 
\end{abstract}

\maketitle


\section{Introduction}
\label{sec:intro}

Let us first consider the problem of finding the resistors in a resistor network from measurements of the power dissipated by each resistor under different loads that are obtained by e.g. imposing voltages at a few terminal nodes. To formulate the problem, let $\cG = (\cV,\cE)$ be an undirected graph, where $\cV$ is a (finite) set of vertices or nodes and $\cE \subset \cV \times \cV$ is the set of edges (we assume there are no loops). By $i \sim j$ we mean that nodes $i$ and $j$ are linked by an edge or in other words, $i \sim j \in \cE$. The resistors are represented by a conductivity function $\sigma : \cE \to (0,\infty)$ and the voltages by a function $u : \cV \to \real$. By Ohm's law, the currents flowing through each resistor under this voltage are $J = \sigma \odot \nabla u$, where $\odot$ is the componentwise or Hadamard product and $\nabla$ is the \emph{discrete gradient}, i.e. the linear operator with $(\nabla u)(i\sim j) = u(i) - u(j)$ for any edge $i\sim j$. To complete this definition we assume that the ordering of the nodes is fixed once and for all for each edge. The particular choice of ordering is irrelevant in the following discussion. We partition the nodes $\cV = \cB \cup \cI$ into ``boundary nodes'' $\cB$, which are the nodes where voltages can be imposed and ``interior nodes'' $\cI$, where the voltage is determined by current conservation or Kirchhoff's node law. Let $f : \cB \to \real$ be a voltage that is imposed on the boundary nodes. The voltage $u$ at all the nodes satisfies the Dirichlet problem
\begin{equation}
    \begin{aligned}
    [\nabla^T ( \sigma \odot \nabla u) ]_{\cI} &= 0,~\text{and}\\
    u_{\cB} &= f,
    \end{aligned}
    \label{eq:dir}
\end{equation}
where subscripting with a set means restriction to the set.  The first equation in \eqref{eq:dir} corresponds to enforcing Kirchhoff's node law at the interior nodes $\cI$, i.e.  that the net current flowing into an interior node is zero.
We may rewrite the first equation in \eqref{eq:dir} as $(L_\sigma u^{(j)} )_{\cI}= 0$, using the \emph{discrete Laplacian} $L_\sigma$ which is defined for any $u : \cV \to \real$ by $ L_\sigma u : \cV \to \real$ as follows (see e.g. \cite{Chung:1997:SGT,Curtis:1998:CPG})
\begin{equation}
    L_\sigma u = \nabla^T ( \sigma \odot \nabla u).
    \label{eq:dlap}
\end{equation}
We may represent $L_\sigma$ as a real $|\cV| \times |\cV|$ symmetric matrix. We denote the submatrices of $L_\sigma$ induced by the partition of the nodes into boundary and interior nodes by e.g. $(L_\sigma)_{\cB\cI}$ for the submatrix of $L_\sigma$ with rows corresponding to boundary nodes and columns corresponding to interior nodes. 

The second equation in \eqref{eq:dir} is analogous to a Dirichlet boundary condition, since the voltages have known values at the boundary nodes $\cB$. Under mild conditions on the graph $\cG$ and the conductivity $\sigma$, the Dirichlet problem \eqref{eq:dir} admits a unique solution for any boundary voltage $f : \cB \to \real$. For example,  one may assume that $\sigma > 0$ and that both $\cG$ and the subgraph induced by the interior nodes are connected, see e.g. \cite{Boyer:2016:SDC}.

 By Joule's law, the power dissipated by each resistor is given by 
 \begin{equation}
    P = J \odot \nabla u = \sigma \odot |\nabla u|^2, \label{eq:joule}
 \end{equation}
 where the square and the absolute value are understood componentwise. The first inverse problem we consider is the \emph{discrete inverse conductivity problem from power measurements}, that is finding the conductivity $\sigma$ given the dissipated power $\sigma \odot |\nabla u^{(j)}|^2$, for $j=1,\ldots,N$, where the voltages $u^{(j)}$ satisfy the Dirichlet problem \eqref{eq:dir} with known voltages at the boundary $f^{(j)} : \cB \to \real$, for $j=1,\ldots,N$.

Measurements of $\sigma \odot |\nabla u^{(j)}|^2$, the power dissipated by the resistors, could be obtained experimentally with e.g. a thermal camera, as a resistor's temperature is expected to increase with increasing dissipated power. A precise model relating resistor temperatures to power dissipated is outside of the scope of this work. The same problem arises if thermal noise induced currents are measured, as we show in \cref{sec:stochastic}. We now recast this inverse problem in a more general framework.

\subsection{Formulation as a redundant system of non-linear equations}
\label{sec:nonlin}
Kuchment and Steinhauer \cite{Kuchment:2012:SIP} and in a more general setting Bal \cite{Bal:2014:HIP} derived local uniqueness results for continuum inverse problems where the measurements consist of internal functionals. Roughly speaking, the inverse problem is recast as a non-linear redundant system of partial differential equations. The injectivity of the redundant system of linear partial differential equations that results from linearization is then determined using Douglis-Niremberg theory \cite{Douglis:1955:IEE}. Here we follow a similar path to prove local uniqueness results for problems defined on graphs. To be more precise, we consider discrete inverse problems that can be written as the non-linear system of equations:
\begin{equation}
\label{eq:nonlin}
 \cL(\gamma,u^{(j)}) = b^{(j)},~j=1,\ldots,N,\\
\end{equation}
where $u^{(j)}$ is a state variable, $\gamma$ is a constitutive parameter,  $b^{(j)}$ is a boundary condition or forcing term and $N$ is the number of different states that are considered. The possibly non-linear operator $\cL$ is such that the system \eqref{eq:nonlin} admits a unique solution $u^{(j)}$ if the $b^{(j)}$ are given. The inverse problem is to recover $\gamma$ from measurements of ``internal functionals''
\begin{equation}
    \cM(\gamma,u^{(j)}) = H^{(j)},~j=1,\ldots,N.
    \label{eq:abstract:ip}
\end{equation}
A key observation in \cite{Kuchment:2012:SIP,Bal:2014:HIP} is that the inverse problem of recovering $\gamma$ from \eqref{eq:abstract:ip} can be thought as the problem of finding $\gamma$ and the $u^{(j)}$ simultaneously, by solving the non-linear system of equations
\begin{equation}
\label{eq:nonlin}
\begin{aligned}
 \cL(\gamma,u^{(j)}) &= b^{(j)},\\
 \cM(\gamma,u^{(j)}) &= H^{(j)},~j=1,\ldots,N.
\end{aligned}
\end{equation}
Linearizing \eqref{eq:nonlin} about a reference $\gamma,u^{(j)}$ we obtain
\begin{equation}
    \label{eq:linearization}
\begin{aligned}
 \partial_\gamma \cL(\gamma,u^{(j)}) \delta\gamma + \partial_u   \cL(\gamma,u^{(j)}) \delta u^{(j)} &= 0,\\
 \partial_\gamma \cM(\gamma,u^{(j)}) \delta \gamma +\partial_u  \cM(\gamma,u^{(j)}) \delta u^{(j)} &= \delta H^{(j)},~j=1,\ldots,N.
 \end{aligned}
\end{equation}
Here $\delta \gamma$ (resp. $\delta u^{(j)}$) is a small perturbation about the reference $\gamma$ (resp. $u^{(j)}$). Also
 $\delta H^{(j)}$ is the difference in measurements between the perturbed and reference configurations, i.e. 
\begin{equation}
    \delta H^{(j)} =  \cM(\gamma+\delta\gamma,u^{(j)}+\delta u^{(j)}) -
 \cM(\gamma,u^{(j)}).
\end{equation}
If we let $v = (\delta\gamma, \{\delta u^{(j)}\})$ then the linear system \eqref{eq:linearization} can be rewritten as
\begin{equation}
    \cA v = \cS,
    \label{eq:linearization:compact}
\end{equation}
for an appropriate matrix $\cA$ and right-hand side $\cS$.

Our main goal here is to study whether the linearized inverse problem \eqref{eq:linearization:compact} admits a unique solution regardless of $\cS$, for several inverse problems with internal functionals defined on graphs. This is sometimes referred to as \emph{local uniqueness} and it is equivalent to proving that $\cA$ is injective or equivalently that $\cA$ has full column rank. If injectivity of $\cA$ holds for a neighborhood of $\gamma$, then the constant rank theorem (see e.g. \cite{Rudin:1976:PMA}) guarantees that the non-linear system admits a unique solution in a neighborhood of $\gamma$. To summarize, we linearize \eqref{eq:nonlin} to obtain \eqref{eq:linearization}. Then we find conditions under which $\delta H^{(j)} = 0$ implies $\delta\gamma=0$ and $\delta u^{(j)} = 0$, $j=1,\ldots,N$, ensuring that $\cA$ is injective.

To fix ideas, let us rewrite the discrete inverse conductivity problem from power measurements in the form \eqref{eq:nonlin}.

\begin{problem}
    \label{prob:dhic}
Find $\gamma = \sigma$ and $u^{(j)}$, $j=1,\ldots,N$ satisfying the non-linear system \eqref{eq:nonlin} where
\begin{equation}
    \cL(\sigma,u) =\begin{bmatrix}
    [\nabla^T ( \sigma \odot \nabla u) ]_{\cI} \\
u_{\cB} 
\end{bmatrix},
~
b^{(j)} = \begin{bmatrix} 0 \\ f^{(j)} \end{bmatrix}
~\text{and}~
\cM(\sigma,u) = \sigma \odot |\nabla u|^2.
\end{equation}
Here the $f^{(j)}$ are known Dirichlet boundary conditions.
\end{problem}

\subsection{Related work}
\label{sec:related}
\Cref{prob:dhic} is a discrete analogue of an inverse problem from internal functionals arising in acousto-electric tomography (also known as ultrasound modulated EIT \cite{Ammari:2008:EIT,Bal:2013:IDK}) and to image conductivities from thermal noise \cite{DeGiovanni:2024:ITN}.  For a review of other inverse problems with internal functionals in the continuum see \cite{Bal:2013:HIP}. We are aware of two other inverse problems on graphs from internal functionals. The first one \cite{Knox:2019:ENP} involves finding $\sigma$ from knowing $|\sigma \odot \nabla u^{(j)}|$. In the second one \cite{Ko:2017:RTI}, the problem is to find $\sigma$ from $\sigma\odot \nabla u^{(j)}$. Both problems fit the framework that we present here, with the caveat that the measurements in \cite{Knox:2019:ENP} have points of non-differentiability due to the absolute value.

\subsection{Contents}
\label{sec:contents}
We start in \cref{sec:stochastic} by showing that \cref{prob:dhic} also arises when measuring thermal noise induced currents in a resistor network. In \cref{sec:realcond} we prove local uniqueness results in the case where the conductivity is real, explaining the necessity of each assumption needed and parallels to the continuum problem \cite{Bal:2014:HIP}. Then we move to the case where the conductivity is complex and we measure the \emph{dissipated power} $\Re(\sigma) \odot |\nabla u^{(j)}|^2$, where $\Re(\sigma)$ is the real part of $\sigma$ (\cref{sec:cplxcond}). Similar uniqueness results can be obtained for a discrete Schr\"odinger problem (\cref{sec:schroe}). A Gauss-Newton approach to solve the non-linear problems is illustrated numerically in \cref{sec:gn}. We conclude with a summary and perspectives in \cref{sec:summary}.

\section{The discrete inverse conductivity problem from thermal noise}
\label{sec:stochastic}
One can arrive to \cref{prob:dhic} with a different experimental procedure where thermally induced random currents are measured at the boundary nodes while selectively heating some of the resistors to a known temperature, as was done in the continuum in \cite{DeGiovanni:2024:ITN}. The thermally induced random currents in a resistor have mean zero and variance that is proportional to the temperature and its conductivity and originate from 
random fluctuations of charge carriers in the resistor. This so-called Johnson-Nyquist noise \cite{Johnson:1928:TAE,Nyquist:1928:TAE} is usually undesirable when designing electrical circuits, but we show that it can be exploited to find the power dissipated inside the network, if we are allowed to control the temperature of the resistors. To  be more precise, let us assume that the temperature (in Kelvin) of the conductors in the network is given by a function $T: \cE \to (0,\infty)$. We model the thermally induced random currents by $J_{\text{noise}} : \cE \to \mathbb{R}$, where for all $e \in \cE$, $J_{\text{noise}}(e)$ is a random variable. The random currents have zero mean and are assumed to be independent of each other, i.e.
\begin{equation}
 \left\langle J_{\text{noise}}(e) \right\rangle = 0 ~\text{and}~
 \left\langle J_{\text{noise}}(e)J_{\text{noise}}(e') \right\rangle = \frac{\kappa}{\pi} \delta_{e,e'} T(e) \sigma(e),~ e,e' \in \cE.
 \label{eq:jdist}
\end{equation}
Here we used angular brackets to denote expectation, the Boltzmann constant $\kappa$ and the Dirac distribution $\delta_{e,e'} =1$ if $e=e'$ and zero otherwise.

Let us assume that the boundary nodes of the network are all connected to the ground, i.e. that $u_{\cB}=0$. If random currents are present, we need to account for them in the current conservation law (first equation of  \eqref{eq:dir}) by adding the net random currents at the interior nodes:
\begin{equation}
 [\nabla^T(\sigma \odot \nabla u)]_{\cI} = (\nabla^T J_{\text{noise}})_{\cI}.
 \label{eq:randomnet}
\end{equation}
The following result gives the mean and covariance of the currents measured at the boundary nodes.
\begin{proposition}
\label{prop:covar}
Let $g = (L_\sigma u)_{\cB}$ be the net currents at the boundary nodes for a voltage $u: \cV \to \mathbb{R}$ satisfying
\begin{equation}
\begin{aligned}
 [\nabla^T(\sigma \odot \nabla u)]_{\cI} &= (\nabla^T J_{\text{noise}})_{\cI},\\
 u_{\cB} &= 0.
\end{aligned}
\label{eq:jrhs}
\end{equation}
For random currents $J_{\text{noise}}$ satisfying \eqref{eq:jdist} and invertible $(L_\sigma)_{\cI\cI}$ we have
\[
\begin{aligned}
 \left\langle g \right\rangle &= 0, ~\text{and}~\\
 \left\langle gg^T \right\rangle  &= \frac{\kappa}{\pi} (L_\sigma)_{\cB\cI} (L_\sigma)_{\cI \cI}^{-1} R_{\cI} \nabla^T \diag( T \odot \sigma) \nabla R_{\cI}^T (L_\sigma)_{\cI \cI}^{-1} (L_\sigma)_{\cI\cB}.
\end{aligned}
\]
\end{proposition}
\begin{proof}
From \eqref{eq:jrhs} we see that
$ (L_\sigma)_{\cI\cB} u_{\cB} + (L_\sigma)_{\cI\cI} u_{\cI} = R_{\cI}\nabla^T J_{\text{noise}}$,
where $R_{\cI}$ is operator restricting a vector the interior nodes. The voltage at the interior nodes is
$
 u_{\cI} = (L_\sigma)_{\cI \cI}^{-1} R_{\cI} \nabla^T J_{\text{noise}}.
$
The net currents at the boundary are
\[
 g = (L_\sigma u)_{\cB} = (L_\sigma)_{\cB\cB} u_{\cB} + (L_\sigma)_{\cB\cI} u_{\cI}
 =  (L_\sigma)_{\cB\cI} (L_\sigma)_{\cI \cI}^{-1} R_{\cI}\nabla^T J_{\text{noise}}.
\]
By linearity of the mean, we clearly have that \[ \langle g \rangle =  (L_\sigma)_{\cB\cI} (L_\sigma)_{\cI \cI}^{-1} R_{\cI} \nabla^T \langle J_{\text{noise}} \rangle = 0.\] The covariance of $g$ follows in a similar manner:
\[
\begin{aligned}
 \langle g g^T \rangle &= (L_\sigma)_{\cB\cI} (L_\sigma)_{\cI \cI}^{-1} R_{\cI} \nabla^T \langle J_{\text{noise}} J_{\text{noise}}^T \rangle \nabla R_{\cI}^T (L_\sigma)_{\cI \cI}^{-1} (L_\sigma)_{\cI\cB}\\
 &= (L_\sigma)_{\cB\cI} (L_\sigma)_{\cI \cI}^{-1} R_{\cI} \nabla^T \diag(\frac{\kappa}{\pi} T \odot \sigma) \nabla R_{\cI}^T (L_\sigma)_{\cI \cI}^{-1} (L_\sigma)_{\cI\cB}.
\end{aligned}
\]
Note that we have used the fact that $L_\sigma$ is symmetric in the previous expression.
\end{proof}
Let $T_0$ be a constant background temperature for the resistors. If we can selectively change the temperature of the specific resistor at edge $e$ by $\delta T$ degrees, the new temperature is $T_e : 
\cE \to \real$  is $T_e = T_0 + \delta T \delta_{e}$, where $\delta_e(e') = \delta_{e,e'}$, for $e' \in \cE$. Let $\langle g_{0} g_0^T \rangle$ be the covariance of the random currents flowing from the boundary nodes to the ground when the temperature of the resistors is $T_0$. Similarly we define the covariances $\langle g_{e} g_e^T \rangle$ for temperatures $T_e$, $e \in \cE$. The following result relates such \emph{differential measurements} (obtained  through $|\cE| + 1$ experiments in which the resistors have prescribed temperatures) to measurements of the power dissipated in the network under known voltages at the boundary. This result is a discrete analogous to \cite[Theorem 3.1]{DeGiovanni:2024:ITN}.
\begin{theorem}
 \label{thm:det}
 Let $u$ solve the Dirichlet problem \eqref{eq:dir}. Then differential measurements of the boundary current covariances and the power dissipated in the network are related as follows
 \begin{equation}
  \label{eq:det}
  [f^T (\langle g_e g_e^T \rangle - \langle g_0 g_0^T\rangle)f]_{e\in\cE} = \frac{\kappa}{\pi} \delta T \sigma \odot |\nabla R_{\cI}^T  R_{\cI} u|^2.
 \end{equation}
\end{theorem}

\begin{proof}
Using \cref{prop:covar} we see that:
\begin{equation}
\label{eq:diff}
\begin{aligned}
    \langle g_e g_e^T \rangle - \langle g_0 g_0^T \rangle & =  \frac{\kappa}{\pi}  (L_\sigma)_{\cB\cI} (L_\sigma)_{\cI \cI}^{-1} R_{\cI} \nabla^T [\delta T \sigma(e) \delta_{e} \delta_{e}^T ]  \nabla R_{\cI}^T (L_\sigma)_{\cI \cI}^{-1} (L_\sigma)_{\cI\cB}\\
    &=  \frac{\kappa}{\pi} \delta T \sigma(e) 
    (\delta_e^T \nabla R_{\cI}^T (L_\sigma)_{\cI \cI}^{-1} (L_\sigma)_{\cI\cB})^T
    (\delta_e^T \nabla R_{\cI}^T (L_\sigma)_{\cI \cI}^{-1} (L_\sigma)_{\cI\cB}).
\end{aligned}
\end{equation}
Therefore, we see that 
\begin{equation}
 f^T(\langle g_e g_e^T \rangle - \langle g_0 g_0^T \rangle) f = \frac{\kappa}{\pi} \delta T \sigma(e) |\delta_e^T \nabla R_{\cI}^T (L_\sigma)_{\cI \cI}^{-1} (L_\sigma)_{\cI\cB} f|^2.
\end{equation}
The result follows from noticing that the quantity in the last equation is proportional to the $e-$th entry of $\sigma \odot | \nabla R_{\cI}^T  R_{\cI} u|^2$.
\end{proof}

Note that \cref{thm:det} only gives information about the power dissipated along resistors spanning two interior nodes. We illustrate \cref{thm:det} numerically in \cref{fig:det}, where we compare the deterministic power dissipated at the resistors between interior nodes to the one obtained from thermal noise induced currents.  Here we used $T_0 = 1$ K and $\delta T = 100$ K and $\sigma = 1$. We also used $\kappa = \pi$ to remove the scaling due to the Boltzmann constant. With $10^4$ realizations, the relative error on this small network was 2.3\%, counting all edges as in \eqref{eq:det} (see \verb|thermal_noise.ipynb| in the supplementary materials \cite{thecode}).

\begin{figure}
 \begin{center}
    \begin{tabular}{c@{}c@{}c}
    \includegraphics[height=4cm]{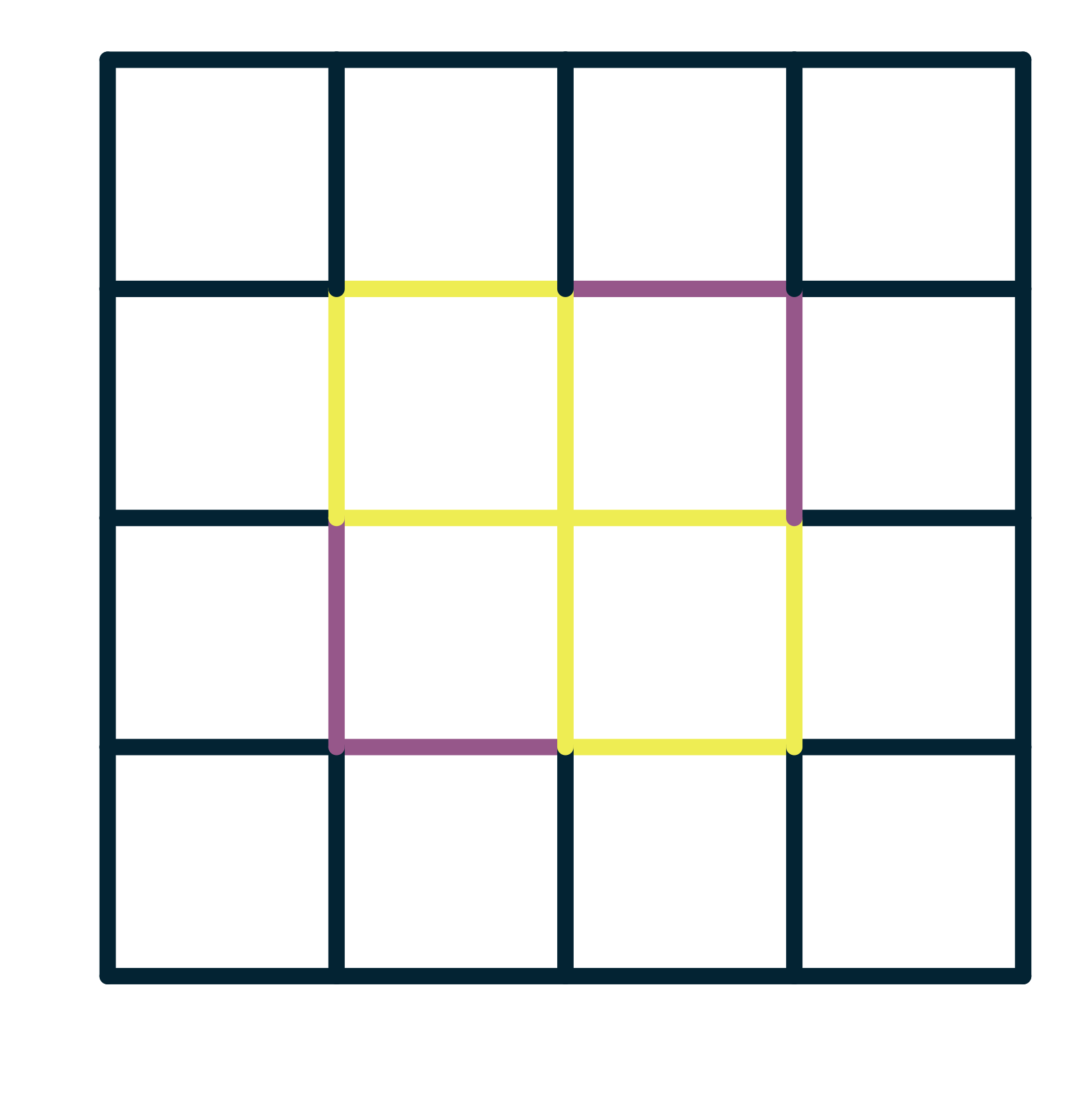} &
    \includegraphics[height=4cm]{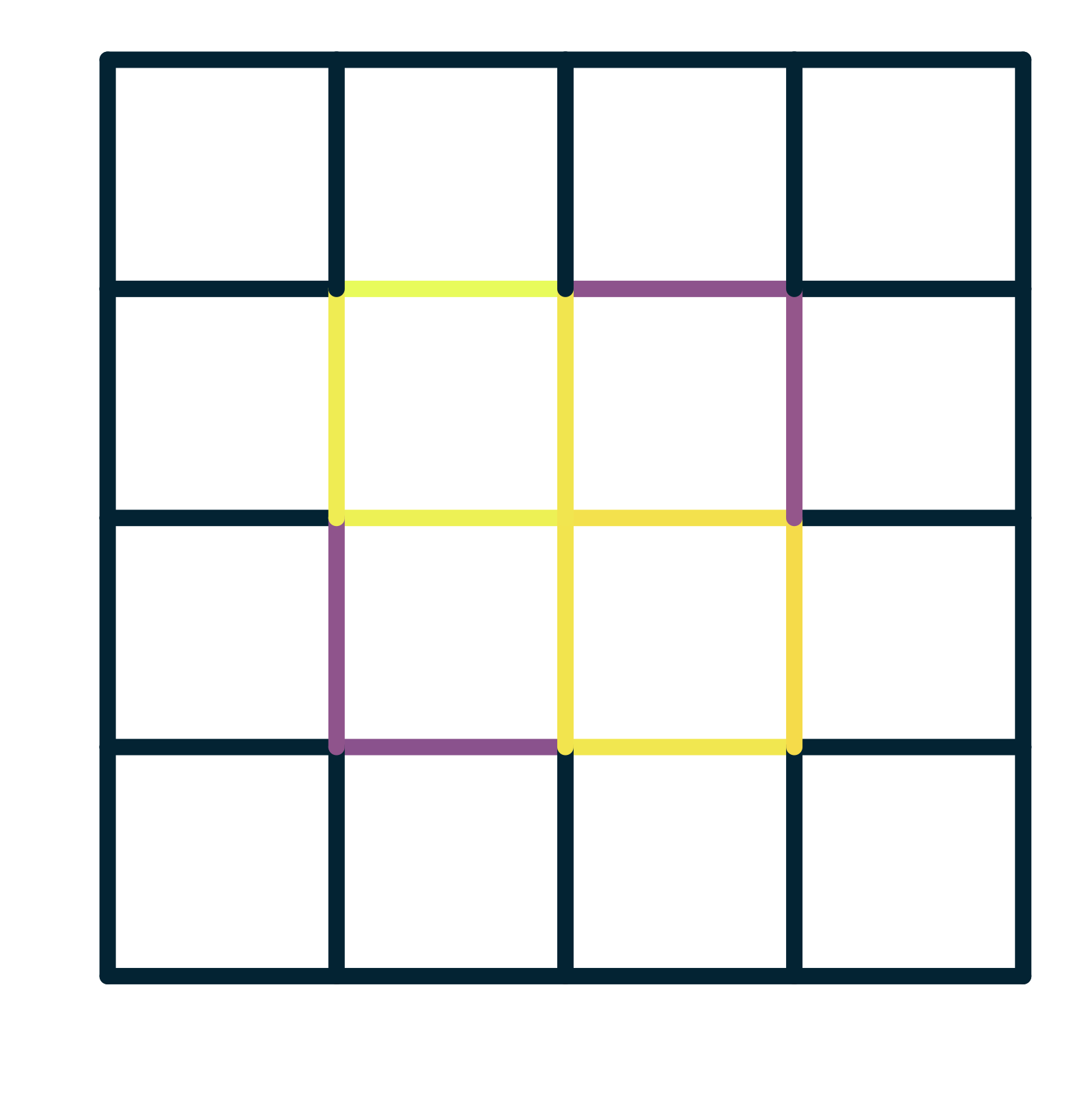} &
    \includegraphics[height=4cm]{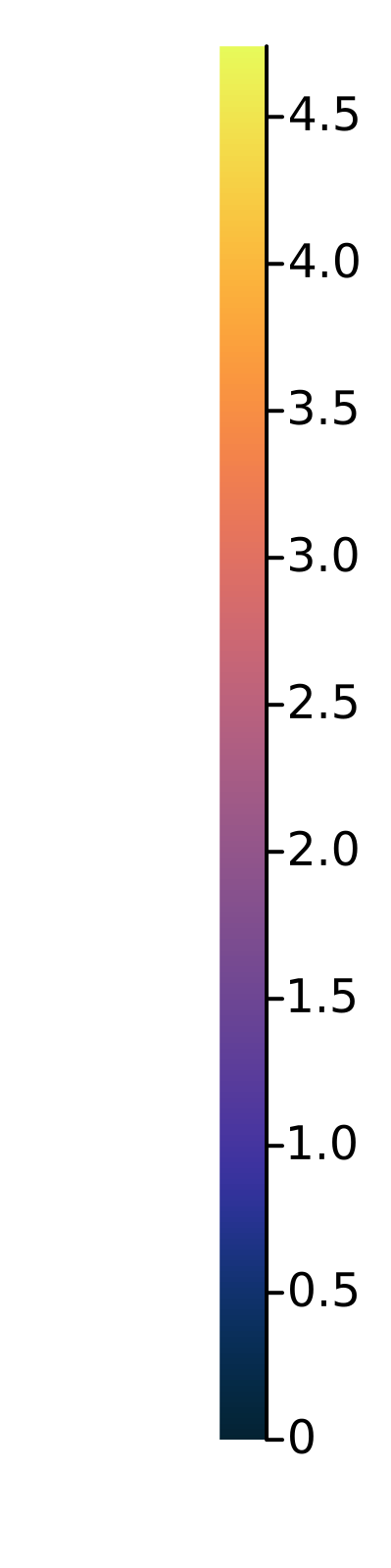}
    \end{tabular}
 \end{center}
 \caption{Comparison of power dissipated (left) to its estimation from thermal noise induced currents (right) using \cref{thm:det}. Here all the edges on edges of the outer square are boundary nodes and the conductivities are all equal to one. The boundary condition corresponds to setting the nodes on the top and right edges to a unit voltage  and the bottom and left edges to a zero voltage.}
 \label{fig:det}
\end{figure}

\section{The real conductivity problem}
\label{sec:realcond}
Recall \cref{prob:dhic} consists in finding a real conductivity $\sigma$ from measurements of dissipated power $\sigma \odot | \nabla u^{(j)}|^2$. To study the local uniqueness we write its linearization as follows 
\begin{equation}
    \begin{aligned}
     |\nabla u^{(j)}|^2 \odot \delta \sigma + 2 \sigma \odot (\nabla u^{(j)}) \odot (\nabla \delta u^{(j)}) &= \delta H^{(j)},\\
     (\delta u^{(j)})_{\cB} & =0,\\
     [ \nabla^T ( (\nabla u^{(j)}) \odot \delta \sigma)]_{\cI}  + [L_\sigma \delta u^{(j)}]_{\cI} &= 0,~\text{for}~ j=1,\ldots,N.
    \end{aligned}
    \label{eq:linsys2}
\end{equation}
The next result gives conditions under which we can guarantee that the linear system \eqref{eq:linsys2} is injective. In other words, if $\delta H^{(j)} = 0$ for $j=1,\ldots,N$, then $\delta \sigma = 0$ and $\delta u^{(j)} = 0$ for $j=1,\ldots,N$. Recall that the support of a function $f: \cX \to \real$ is $\supp f = \{ x \in \cX ~|~ f(x) \neq 0 \}$. Moreover, the characteristic or indicator function of some a subset $\cX$ of $\cY$ is $\ifun_\cX : \cY \to \{0,1\}$ with $\ifun_\cX(x)= 1$ if $x\in \cX$ and zero otherwise.
\begin{theorem}
\label{thm:realcond}
Define the sign functions $s^{(j)}: \cE \to \{-1,1\}$ by 
\begin{equation}
    s^{(j)} = 1-2\ifun_{\supp \nabla u^{(j)}}.
    \label{eq:sj}
\end{equation}
Assume that the following conditions hold:
\begin{enumerate}[i.]
 \item The supports of the $\nabla u^{(j)}$ cover $\cE$, i.e. $\cup_{j=1}^N \supp \nabla u^{(j)} = \cE$.
 \item The matrix $[L_\sigma]_{\cI\cI}$ is invertible.
 \item The matrices $[L_{s^{(j)} \odot \sigma}]_{\cI\cI}$ are invertible for $j=1,\ldots,N$.
\end{enumerate}
Then the linear system \eqref{eq:linsys2} admits a unique solution, or in other words the linearization of \cref{prob:dhic} is injective.
\end{theorem}
Let us introduce some notation before proceeding with the proof. 
The linear operator restricting a function $u : \cV \to \real$ to the interior nodes $\cI$ by $R_{\cI}$ (i.e.  $R_{\cI}u = u_{\cI}$). Its adjoint $R_{\cI}^T$ takes a function defined on $\cI$ and extends it by zeroes to all of $\cV$. Finally, if $\cX$ is a finite set and $f : \cX \to \real$, we define its pseudoinverse $f^\dagger$ so that $f \odot f^\dagger = f^\dagger \odot f = \ifun_{\supp f}$. Componentwise, the pseudoinverse of $f$ is given for $x \in \cX$ by
\begin{equation}
    f^\dagger(x) = \begin{cases} f(x)^{-1}, & \text{if $f(x)\neq 0$,}\\
    0 & \text{otherwise}.
    \end{cases}
    \label{eq:dagger}
\end{equation}
\begin{proof}
To prove injectivity, we need to show that if we take \eqref{eq:linsys2} and set $\delta H^{(j)} = 0$, $j=1,\ldots,N$ then we can conclude that $\delta\sigma =0$ and $\delta u^{(j)} = 0$, $j=1,\ldots,N$. Assumption (ii) is needed to ensure that the Dirichlet problem \eqref{eq:dir} admits a unique solution.

From the second equation of \eqref{eq:linsys2}  we see that $\delta u^{(j)}_{\cB}=0,~j=1,\ldots,N$.
Thus, for a fixed $j \in \{0,\ldots,N\}$, we can see from  the first equation of \eqref{eq:linsys2} that
\begin{equation}
 \delta\sigma \odot \ifun_{\supp \nabla u^{(j)}} =-2(\nabla u^{(j)})^{\dagger}\odot \sigma\odot(\nabla R_{\cI}^{T}\delta u_{\cI}^{(j)}).
 \label{eq:deltasig2}
\end{equation}
Now notice that in the third equation of \eqref{eq:linsys2}, $\delta\sigma$ only appears multiplied by $\nabla u^{(j)}$. Therefore, we can replace $\delta\sigma$ by $\delta\sigma \odot \ifun_{\supp \nabla u^{(j)}}$ (for that particular $j$). Hence, using \eqref{eq:deltasig2} and the third equation of \eqref{eq:linsys2} we get
\begin{equation}
 -2R_{\cI}\nabla^{T}[
 \nabla u^{(j)}\odot 
 (\nabla u^{(j)})^{\dagger} \odot 
 \sigma \odot 
 (\nabla R_{\cI}^{T}\delta u_{\cI}^{(j)})
 ]
 +
 R_{\cI}\nabla^{T}[
 \sigma \odot 
 (\nabla R_{\cI}^{T}\delta u_{\cI}^{(j)})
 ]
 =0.
\end{equation}
Clearly this leads to
\begin{equation}
 R_{\cI}\nabla^{T}[
 (1-2\nabla u^{(j)}\odot 
 (\nabla u^{(j)})^{\dagger} )\odot 
 \sigma \odot 
 (\nabla R_{\cI}^{T}\delta u_{\cI}^{(j)})
 ]
 =0.
 \label{eq:thirdrow2}
\end{equation}
By using $s^{(j)} = 1-2\nabla u^{(j)}\odot 
 (\nabla u^{(j)})^{\dagger}$ in \eqref{eq:thirdrow2} we obtain
\begin{equation}
 [L_{s^{(j)} \odot \sigma}]_{\cI\cI} \delta u_{\cI}^{(j)} = 0.
 \label{eq:thirdrow3}
\end{equation}
By assumption (iii), we see that $\delta u_{\cI}^{(j)} = 0$ and that the entries in $\delta \sigma$ associated with edges in $\supp \nabla u^{(j)}$ are zero. By repeating the same procedure for all experiments $j = 1,\ldots,N$, we can conclude  that $ \delta u^{(j)} = 0$. Finally, \eqref{eq:deltasig2} and assumption (i) imply that $\delta \sigma = 0$, which is the desired result.
\end{proof}

\subsection{Parallels with the continuum}
Since $\sigma >0$, it is possible to show that $L_\sigma$ is symmetric positive semidefinite. On the other hand, the matrices $L_{s^{(j)} \odot \sigma}$ are indefinite. Indeed, such matrices correspond to networks where the conductivity is allowed to change signs. There is a striking similarity between our approach and the continuum approach in \cite{Bal:2014:HIP}. Indeed, that $L_\sigma$ defines a positive quadratic form can be seen as a discrete version of the classic notion of ellipticity of second order partial differential operators (see e.g. \cite{Evans:1998:PDE}). In Fourier space, this corresponds to having a positive (or negative) definite quadratic in the spatial frequency parameter. Assumption (iii) in \cref{thm:realcond} requires that an indefinite quadratic form be non-singular and is analogous to the notion of ellipticity in the sense of Douglas-Nirenberg that is used in \cite{Bal:2014:HIP}. In Fourier space, this requires that the principal symbol of a differential operator can vanish only when the spatial frequency component is zero.

\Cref{prob:dhic} can be thought of as a discrete analogue of Acousto-Electric Tomography. It is known (see e.g. \cite{Capdeboscq:2009:IMN}) that in two dimensions at least two different boundary conditions are required to obtain unique reconstruction. On the other hand \cref{thm:realcond} shows that if the gradient of the reference potential does not vanish, then \emph{one single boundary condition} could be sufficient to determine $\delta \sigma$. This may seem contradictory at first, as it would indicate that in a 2D finite difference discretization of the AET problem we can recover $\sigma$ with one single boundary condition. In such a discretization we would need to define difference operators $D_1$ and $D_2$ for the horizontal and vertical edges, so that the data available is $\sigma \odot (| D_1 u|^2 + |D_2 u|^2)$, where $u$ is a suitable discretization of the potential. Since this is less data than the data available for \cref{prob:dhic}, we expect that it may take more measurements, see also the numerical scheme in \cite{DeGiovanni:2024:ITN}.

\subsection{Examples}
\label{sec:examples}
We illustrate the necessity of the assumptions in \cref{thm:realcond} with examples. We start with a graph where one single Dirichlet boundary condition ($N=1$) is needed to guarantee local injectivity.

\begin{figure}
\begin{center}
\includegraphics[width=0.9\textwidth]{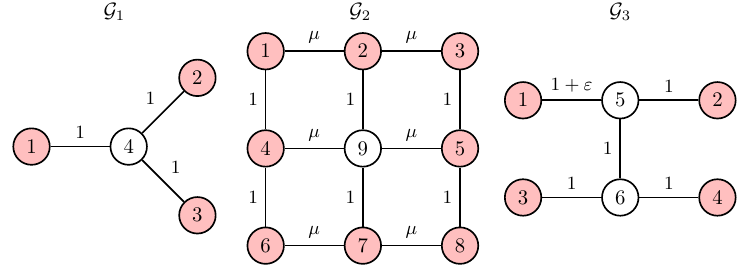}
\end{center}
\caption{Examples of weighted graphs illustrating different cases of \cref{thm:realcond}, where the conductivities are indicated on each edge. For each graph $\cV \subset \mathbb{N}$, and the nodes are ordered so that the boundary nodes appear first (in red) and the interior nodes last (in white). To ensure positive edge weights or conductivities it is assumed that $\mu > 0$ and $|\eps|<1$. For more details see \cref{sec:examples}.}
\label{fig:examples}
\end{figure}

\begin{example} \label{ex:g1}
Consider the network $\cG_1$ in \cref{fig:examples}. If we impose the boundary voltage $f(1)=1$, $f(2)=f(3)=0$, then $u(4) = 1/3$ and thus $\supp \nabla u = \cE$. Therefore, \cref{thm:realcond} applies, and we can reconstruct the conductivity locally from a single boundary condition (i.e. with $N=1$), see \texttt{example2.ipynb} in supplementary materials \cite{thecode}.
\end{example}

The next example shows a network where two Dirichlet boundary conditions are needed to obtain local uniqueness and where local uniqueness does not hold in a degenerate case.
\begin{example} \label{ex:g2}
Consider the network $\cG_2$ in \cref{fig:examples} and the boundary conditions $f^{(1)}$ and $f^{(2)}$ given as follows.
\begin{center}
\begin{tabular}{c|cccccccc}
 $i$   & $1$ & $2$ & $3$ & $4$ & $5$ & $6$ & $7$ & $8$\\\hline
 $f^{(1)}(i)$ & 1 & 1 & 1 & 1/2 & 1/2 & 0 & 0 & 0\\
 $f^{(2)}(i)$ & 1 & 1/2 & 0 & 1 & 0 & 1 & 1/2 & 0
\end{tabular}
\end{center}

Since there is a single interior node (number $9$), a simple calculation reveals that $u^{(1)}(9) = u^{(2)}(9) = 1/2$. Hence, currents flow only along the vertical (resp. horizontal) edges for $f^{(1)}$ (resp. $f^{(2)}$). In other words, $\supp \nabla u^{(1)}$ (resp. $\supp \nabla u^{(2)}$) consists of the vertical edges (resp. horizontal edges). Thus, $\supp \nabla u^{(1)} \cup \supp \nabla u^{(2)} = \cE$, which ensures that \cref{thm:realcond} assumption (i) is satisfied with $N=2$ boundary conditions. 
Moreover, we can calculate $[L_\sigma]_{\cI\cI} = 2 (1+\mu)$, which ensures \cref{thm:realcond} assumption (ii) holds for any $\mu >0$. The sign functions from \eqref{eq:sj} are given by $s^{(1)} = - s^{(2)}$ with $s^{(1)}$ negative on the vertical edges and positive on the horizontal edges. Therefore, we see  that $[L_{s^{(1)}\odot \sigma}]_{\cI\cI} = -[L_{s^{(2)}\odot \sigma}]_{\cI\cI} = 2 (-1+\mu)$.  We get local uniqueness as long as $\mu \neq 1$, which ensures that \cref{thm:realcond} assumption (iii) holds. When $\mu = 1$, assumption (iii) is not satisfied and thus local uniqueness does not hold. Indeed, consider conductivity perturbations of the form $\delta \sigma = c_1 v_1 + c_2 v_2$ where $c_1, c_2$ are scalars and the only non-zero entries of $v_1$ and $v_2$ are  $v_1(2 \sim 9) = - v_1(9 \sim 7) = 1$, $v_2(4 \sim 9) = - v_2(9 \sim 5) = 1$. Any conductivity perturbation of this form leads to zero change in the linearized data (see \texttt{example3.ipynb} in supplementary materials \cite{thecode}). We recall however that lack of injectivity of the linearization of a mapping does not necessarily mean the mapping is not invertible.
\end{example}

Next we illustrate the possibility of having an ill-conditioned inverse problem by constructing a network and boundary condition where one edge has an arbitrarily small current flowing through it, and therefore the power measurements are not a reliable way of measuring its conductivity.
\begin{example}
We consider here the network $\mathcal{G}_3$ in \cref{fig:examples} where local uniqueness holds, but the resulting linear system can be made arbitrarily ill-conditioned. For this example we use $N=1$, $f^{(1)}(1) = f^{(1)}(2) =1$ and $f^{(1)}(3) = f^{(1)}(4) = 0$. Lengthy but straightforward calculations give that the solution $u$ to the corresponding Dirichlet problem must satisfy
\[
u(5) = \frac{4+3\eps}{8 + 3\eps}
~\text{and}~
u(6) = \frac{4+2\eps}{8+3\eps}.
\]
Therefore, $|(\nabla u)(5,6)| = |\eps / (8+3\eps)| = \mathcal{O}(\varepsilon)$ as $\varepsilon \to 0$.  Since we divide by the non-zero components of $\nabla u$ in \eqref{eq:deltasig2}, we expect that the conditioning of the problem is inversely proportional to $\varepsilon$, as can be confirmed numerically (see \texttt{example3.ipynb} in supplementary materials \cite{thecode}).
\end{example}

\section{The complex conductivity problem}
\label{sec:cplxcond}

We assume complex conductivities of the form $\sigma = \sigma' + \jmath\omega \sigma''$, where $\sigma' \in (0,\infty)^{\cE}$, $\sigma'' \in \real^{\cE}$, $\omega \in \real$ is a known frequency and $\jmath=\sqrt{-1}$.  Such complex $\sigma$ are approximations to the admittance (the reciprocal of the impedance) of a general passive circuit element \cite{Bott:1949:ISW}. The Dirichlet problem for this class of complex conductivities is as follows
\begin{equation}
 \begin{aligned}
    \left[ \nabla^T(\sigma(\omega) \odot \nabla u(\omega)) \right]_{\cI} &= 0,\\
    [u(\omega)]_{\cB} &= f(\omega),
 \end{aligned}
 \label{eq:dir:sig:cplx}
\end{equation}
where $f(\omega) \in \complex^{\cB}$ is the known data at the boundary. This problem can be shown to admit a unique solution in particular  when $\sigma' > 0$ and both $\cG$ and the subgraph induced by the interior nodes are connected, see e.g. \cite[Theorem 4.1]{Boyer:2016:SDC}. We assume the \emph{dissipated power} (or real power) $\sigma' \odot |\nabla u(\omega)|^2$ is measured at each electric element.

\subsection{Inverse problem from dissipated power measurements}
In this case a single frequency measurement is not enough to determine the complex conductivity $\sigma(\omega)$. Indeed, the number of unknowns is $2|\cE|$, but the number of equations is $|\cE|$. Thus, we consider the problem of finding the conductivity from measurements at \emph{two frequencies} $\omega_0$ and $\omega_1$. In order to capitalize on the results for the real case (\cref{thm:realcond}), we assume that $\omega_0 = 0$ and $f_0 \equiv f(\omega_0) \in \real^{\cB}$. We consider both the conductivity and the voltages as unknowns. Since $u_0 \equiv u(\omega_0)$ is real but $u_1 \equiv u(\omega_1)$ is complex, we consider $u_1$ and $\overline{u_1}$ as independent variables. Putting everything together we can formulate the inverse problem of finding the complex conductivity from a static (zero frequency measurement) and a non-zero frequency measurement as follows.
\begin{problem}
\label{prob:cplxcond} 
Find $\gamma = (\sigma',\sigma'')$ and $u^{(j)}=(u_0^{(j)},u_1^{(j)},\overline{u_1}^{(j)})$,  for $j=1,\ldots,N$  in the non-linear system \eqref{eq:nonlin} where
\begin{equation}
    \cM(\gamma,u) = \begin{bmatrix}
    \sigma' \odot |\nabla u_0|^2\\
    \sigma' \odot |\nabla u_1|^2
\end{bmatrix},
~
    \cL(\gamma,u) = \begin{bmatrix}
        \left[\nabla^T(\sigma' \odot \nabla u_0)\right]_{\cI} \\
        \left[\nabla^T((\sigma' +\jmath\omega_1\sigma'') \odot \nabla u_1)\right]_{\cI} \\
        \left[\nabla^T((\sigma' -\jmath\omega_1\sigma'') \odot \nabla \overline{u_1})\right]_{\cI} \\
        [u_0]_{\cB}\\
        [u_1]_{\cB}\\
        [\overline{u_1}]_{\cB}\\
    \end{bmatrix},
\end{equation}
and 
\begin{equation}
    b^{(j)} = \begin{bmatrix}
        0;
        0;
        0;
        f_0^{(j)};
        f_1^{(j)};
        \overline{f_1}^{(j)}
    \end{bmatrix}, 
\end{equation}
where a semicolon denotes vector concatenation.
Here the $f_0^{(j)}$ (resp. $f_1^{(j)}$) are known real (resp. complex) Dirichlet boundary conditions and $\omega_1$ is known.
\end{problem}
The number of unknowns in \cref{prob:cplxcond} is $2|\cE|+3N|\cV|$ and the number of equations is $2N|\cE| + 3 N |\cV|$. Notice that for $N=1$, the resulting system of non-linear equations is formally determined (i.e. same number of equations as unknowns), however as we see next, the resulting Jacobian is unlikely to be invertible. We give a sufficient condition guaranteeing that the Jacobian of the non-linear system is injective.
\begin{theorem} \label{thm:cplxcond}
 Assume that the graph $\cG$, $\sigma'$ and the boundary conditions $f^{(j)}$ are such that \cref{thm:realcond} is satisfied. Define for $j=1,\ldots,N$ the $|\cE| \times |\cE|$ complex matrices
\begin{equation}
    \label{eq:aj}
    A^{(j)} = -\jmath\diag(\overline{\nabla u_1^{(j)}} ) \nabla R_{\cI}^T {[L_{\sigma'+\jmath\omega_1\sigma''}]}_{\cI\cI}^{-1} R_{\cI}\nabla^T\diag(\nabla u_1^{(j)}).
\end{equation}
 If $\Re([A^{(1)}; \ldots; A^{(N)}])$ is full rank\footnote{Here the imaginary part of a matrix is given by the matrix of imaginary parts of its elements. The semicolon denotes vertical concatenation of matrices.}, then the linearization of the two frequency inverse complex conductivity  \cref{prob:cplxcond} is injective.
\end{theorem}

\begin{proof}
  Linearizing the equations in the system that are associated with $u_0^{(j)}$ and equating to zero we get
 \begin{equation}
    \begin{aligned}
    \delta\sigma' \odot |\nabla u_0^{(j)}|^2 + 2\sigma' \odot \nabla \delta u_0^{(j)} \odot \nabla u_0^{(j)} &= 0,\\
     [\delta u_0^{(j)}]_{\cB} &=0,\\
     [L_{\delta\sigma'} u_0^{(j)} + L_{\sigma'} \delta u_0^{(j)} ]_{\cI} &= 0.
    \end{aligned}
    \label{eq:complexsystem}
\end{equation}
 By \cref{thm:realcond} we conclude that $\delta\sigma' = 0$ and $\delta u_0^{(j)} = 0$, $j=1,\ldots,N$.
Setting the linearized equations at the interior nodes involving $u_1^{(j)}$ we get:
\begin{equation}
 [ L_{\delta\sigma'} u_1^{(j)}]_{\cI} + \jmath\omega_1 [L_{\delta\sigma''} u_1^{(j)}]_{\cI}
 + [L_{\sigma'+\jmath \omega \sigma''}]_{\cI} = 0.
\end{equation}
Now, we  can use that $\delta\sigma' = 0$ and $[\delta u_1^{(j)}]_{\cB} =  0$ to solve for $[\delta u_1^{(j)}]_{\cI}$ in terms of $\delta\sigma''$. The process can be repeated for the linearized equations at the interior nodes involving $\overline{u_1}^{(j)}$, which reveals that
\begin{equation}
\begin{aligned}
    \delta u_1^{(j)} &= -\jmath\omega_1  R_{\cI}^T {[L_{\sigma'+\jmath\omega_1\sigma''}]}_{\cI\cI}^{-1} R_{\cI}\nabla^T\diag(\nabla u_1^{(j)})\delta\sigma'',~\text{and}\\
    \overline{\delta u_1}^{(j)} &= \jmath\omega_1  R_{\cI}^T {[L_{\sigma'-\jmath\omega_1\sigma''}]}_{\cI\cI}^{-1} R_{\cI}\nabla^T\diag(\overline{\nabla u_1}^{(j)})\delta\sigma''.\\
\end{aligned}
\label{eq:deltau1jcond}
\end{equation}
Linearizing the equations involving the dissipated power measurements for $\omega_1$ and equating to zero we obtain:
\begin{equation}
  \delta \sigma' \odot |\nabla u_1^{(j)}|^2 + \sigma' \odot \nabla \overline{u_1}^{(j)} \odot \nabla \delta u_1^{(j)} + \sigma' \odot \nabla u_1^{(j)} \odot \nabla \overline{\delta u_1}^{(j)} = 0.
\end{equation}
Using that $\delta\sigma'=0$ and the expressions for $\delta u_1^{(j)}$ and $\overline{\delta u_1}^{(j)}$ in \eqref{eq:deltau1jcond} we obtain after a few manipulations that
\begin{equation}
\begin{aligned}
      -&\jmath\diag(\overline{\nabla u_1}^{(j)} ) \nabla R_{\cI}^T {[L_{\sigma'+\jmath\omega_1\sigma''}]}_{\cI\cI}^{-1} R_{\cI}\nabla^T\diag(\nabla u_1^{(j)}) 
     \delta \sigma '' \\
     + &\jmath\diag(\nabla u_1^{(j)} ) \nabla R_{\cI}^T {[L_{\sigma'-\jmath\omega_1\sigma''}]}_{\cI\cI}^{-1} R_{\cI}\nabla^T\diag(\overline{\nabla u_1}^{(j)})
     \delta \sigma''= 0,
\end{aligned}
\end{equation}
or in other words $ A^{(j)} \delta\sigma'' + \overline{A}^{(j)} \delta\sigma'' = 2\Re(A^{(j)}) \delta\sigma'' = 0$. By assumption the matrix $\Re[A^{(1)};\ldots;A^{(N)}]$ is full rank. We thus reach the conclusion $\delta\sigma'' =0$. 
\end{proof}

\begin{remark}
 Notice that the complex matrices $A^{(j)}$, $j=1,\ldots,N$ are singular. Indeed, the current conservation law (first equation of \eqref{eq:dir:sig:cplx}) implies that 
\begin{equation}
    \label{eq:ajnull}
    A^{(j)}(\sigma' + \jmath \omega_1 \sigma'') = 0.
\end{equation}
This does not rule out the assumption in \cref{thm:cplxcond}, since the rank of a complex matrix may be different from that of its real or imaginary parts. For example, it is possible to have a singular complex matrix with a non-singular real part or vice-versa, a non-singular complex matrix with singular real part.
The dimensions of the matrices used to define $A^{(j)}$ do reveal that $\Re A^{(j)}$ has rank at most $\max(|\cE|,|\cI|)$. Thus, at least for planar graphs with two boundary nodes, we cannot get injectivity of the linearized problem with $N=1$. Indeed, if $\cG$ is planar then Euler's theorem holds:
\begin{equation}
 2 = |\cV| - |\cE| + |\cF| = |\cI| - |\cE| + |\cB| + |\cF|,
\end{equation}
where $|\cF|$ is the number of faces and $|\cF|\geq 1$. Hence, $|\cI| < |\cE|$ and thus $\Re(A^{(j)})$ in its own is singular.
\end{remark}

\begin{example}
    \label{ex:cplxcond}
     Consider the graph $\cG_3$ in \cref{fig:examples} with $\omega_1 = 1$ and real and imaginary conductivities given as in the following table.
    \begin{center}
    \begin{tabular}{l|c|c|c|c|c}
     $i\sim j$   & $1\sim 5$ & $2\sim 5$ & $3\sim 6$ & $4\sim 6$ & $5\sim 6$\\\hline
     $\sigma'(i\sim j)$ & 2 & 1 & 1 & 3 & 1\\
     $\sigma''(i\sim j)$ & 2 & 1 & 1 & 2 & 3
    \end{tabular}
    \end{center}
    We computed the Jacobian $\cA$ of \cref{prob:cplxcond} when the Dirichlet boundary conditions are $\delta_1,\ldots,\delta_N$, for $N=1,\ldots,4$.  We report in the table below the dimensions of the Jacobian, its rank and condition number. We also include the rank of the complex matrix $A \equiv [A^{(1)};\ldots;A^{(N)}]$ and its real part as they are needed for the injectivity condition of \cref{thm:cplxcond}. Recall that the condition number of rectangular matrices is defined by the ratio of the largest singular value to the smallest one, and it is a numerically stable of determining whether a matrix is full rank, see e.g. \cite{Golub:2013:MC}. 
 
    \begin{center}
    \begin{tabular}{c|c|c|c|c|c}
    $N$ & $\cA$ dimensions & $\text{rank}\,\cA$	& $\text{cond}\,\cA$	& $\text{rank}\,A$ & $\text{rank}\,\Re A$\\\hline
    1	& (28, 28)	& 27 &	1.54951e17	& 2	& 4\\
    2	& (56, 46)	& 45 &	1.36094e16	& 3	& 4\\
    3	& (84, 64)	& 64 &	3319.42	    & 4	& 5\\
    4	& (112, 82)	& 82 &	1837.58	    & 4	& 5\\
    \end{tabular}
    \end{center}
    To determine whether the Jacobian $\cA$ is full rank we need to check whether its rank is equal to its smaller dimension (the number of columns), which only happens for $N=3$ and $N=4$. This is consistent with the rank of $\Re A$ being equal to the number of edges ($|\cE| = 5$ in this case). Notice that in these experiments $\text{rank}\, A < \text{rank}\, \Re A$, which could be attributed to the complex conductivity being always in the nullspace of $A$, see \eqref{eq:ajnull}. Finally, we observe that the condition number of $\cA$ improves with $N$. This table can be reproduced with \verb|example4.ipynb| in the supplementary materials \cite{thecode}.

\end{example}
\section{The inverse Schr\"odinger problem}
\label{sec:schroe}
Consider a graph $\cG = (\cV,\cE)$, where the vertices are partitioned into boundary nodes $\cB$ and interior nodes $\cI$. The (discrete) \emph{Schr\"odinger problem} with conductivity $\sigma \in \complex^{\cE}$,  \emph{Schr\"odinger potential} $q \in \complex^{\cI}$ and Dirichlet boundary condition $f \in \complex^{\cB}$ consists in finding $u \in \complex^{\cV}$ satisfying the Dirichlet problem 
\begin{equation}
    \label{eq:qdir}
    \begin{aligned}
       (L_{\sigma}u)_{\cI} +  q \odot u_{\cI}  &= 0,\\
        u_{\cB}  &= f.
    \end{aligned}
\end{equation}
Clearly if $(L_{\sigma})_{\cI\cI}+ \diag(q)$ is invertible, then the Schr\"odinger problem admits a unique solution $u$ for any boundary condition $f$ which is given by
\begin{equation}
    u_{\cB} = f ~\text{and} ~ u_{\cI} = -[(L_{\sigma})_{\cI\cI}+\diag(q)]^{-1}(L_{\sigma})_{\cI\cB}f.
\end{equation}

To motivate this problem consider the case where $\sigma \in (0,\infty)^{\cE}$ and $q \in (0,\infty)^{\cI}$. Here solving \eqref{eq:qdir} is equivalent to finding the voltages in a resistor network with conductors on each edge of $\cG$ with conductivities given by $\sigma$ and conductors connecting each interior node to the ground (zero potential), with conductivities given by $q$. In this case $q\odot u_{\cI}$ corresponds to currents leaking to the ground. We emphasize that the problem \eqref{eq:qdir} is more general, since the Schr\"odinger potential could have negative or complex valued entries. 

We consider an inverse problem where the goal is to recover the Schr\"odinger potential $q$ from measurements of the dissipated power $\Re(q) \odot |u_{\cI}|^2$. This is a slight abuse of nomenclature, since this quantity does not physically correspond to power if $\Re(q(x)) <0$ for some interior node $x$.  In \cref{sec:realq} we study the case where $q$ is real and then in \cref{sec:cplxq} we generalize our results to the case where $q$ is allowed to be complex.

\subsection{Real Schr\"odinger problem}
\label{sec:realq}
Here we consider the case where $\sigma \in (0,\infty)^{\cE}$ and $q \in \real^{\cI}$. This gives rise to the following inverse problem.
\begin{problem}
    \label{prob:realq}
    Find $q \in \real^{\cI}$ and $u^{(j)}$ for $j=1,\ldots,N$ in the non-linear system \eqref{eq:nonlin} where $\gamma = q$, 
    \begin{equation}
        \cL(q,u) = \begin{bmatrix}
            (L_\sigma u)_{\cI} + q \odot u_\cI\\
            u_{\cB}
        \end{bmatrix},
        ~
        b^{(j)} = 
        \begin{bmatrix}
            0\\
            f^{(j)}
        \end{bmatrix}
        ~\text{and}~
        \cM(q,u) = q\odot |u_{\cI}|^2.
    \end{equation}
    Here the $f^{(j)}$ are known real Dirichlet boundary conditions.
\end{problem}
The result analogous to \cref{thm:realcond} for the real Schr\"odinger case is as follows.
\begin{theorem}
    \label{thm:realq}
    Define the sign functions $s^{(j)}: \cI \to \{-1,1\}$ by
    \begin{equation}
        \label{eq:sjq}
        s^{(j)} = 1 - 2 \ifun_{\supp u_{\cI}^{(j)}},~j=1,\ldots,N.
    \end{equation}
    Assume that the following conditions hold:
    \begin{enumerate}[i.]
        \item The supports of $u_{\cI}^{(j)}$ cover $\cI$, i.e. $\cup_{j=1}^N \supp u_{\cI}^{(j)} = \cI$.
        \item The matrix $[L_\sigma]_{\cI\cI} + \diag(q)$ is invertible.
        \item The matrices $[L_\sigma]_{\cI\cI} + \diag(s^{(j)}\odot q)$ are invertible for $j=1,\ldots,N$.
    \end{enumerate}
    Then the linearization of \cref{prob:realq} is injective.
\end{theorem}
\begin{proof}
The proof is very similar to the proof of \cref{thm:realq}, we take the linearization of the \cref{prob:realq}, equate it to zero and show that the only way this can happen is if $\delta q=0$ and $\delta u_{\cI}^{(j)} = 0$, $j=1,\ldots,N$. Assumption (ii) is needed to ensure that the Dirichlet problem admits a unique solution. Now let us proceed by linearizing \cref{prob:realq} and equating to zero to obtain for $j=1,\ldots,N$ that
\begin{equation}
 \begin{aligned}
    (L_{\sigma} \delta u^{(j)})_{\cI} +\delta q \odot u_{\cI}^{(j)} + q \odot \delta u_{\cI}^{(j)}  &= 0,\\
    (\delta u^{(j)})_{\cB}  &= 0 ~\text{and}~\\
    \delta q |u_{\cI}^{(j)}|^{2}+2q \odot \delta u_{\cI}^{(j)} \odot u_{\cI}^{(j)} &= 0.
 \end{aligned}
  \label{eq:linsysq}
\end{equation}
Using the third equation in \eqref{eq:linsysq}  we see that:
\begin{equation}
    \delta q \odot \ifun_{\supp u_{\cI}^{(j)}} = -2 (u_{\cI}^{(j)})^\dagger \odot  q \odot \delta u_{\cI}^{(j)},~j=1,\ldots,N.
    \label{eq:deltaq}
\end{equation}
Now in the first equation of \eqref{eq:linsysq}, $\delta q$ appears only multiplied by $u_{\cI}^{(j)}$. So we can replace $\delta q$ by $\delta q \odot \ifun_{\supp u_{\cI}^{(j)}}$ in the first equation of \eqref{eq:linsysq}, for that particular $j$. Using that $(\delta u^{(j)})_{\cB}  = 0$ (second equation of \eqref{eq:linsysq}) and \eqref{eq:deltaq} in the first equation of \eqref{eq:linsysq} we get 
\begin{equation}
  ((L_\sigma)_{\cI\cI} + \diag(q)) \delta u_{\cI}^{(j)} - 2 (u_{\cI}^{(j)})^\dagger \odot u_{\cI}^{(j)} \odot q \odot \delta u_{\cI}^{(j)} = 0.
\end{equation}
Recognizing that $(u_{\cI}^{(j)})^\dagger \odot u_{\cI}^{(j)} = \ifun_{\supp u_{\cI}^{(j)}}$, we get
\begin{equation}
    ((L_{\sigma})_{\cI\cI}  + \diag((1 -2 \ifun_{\supp u_{\cI}^{(j)}})\odot q)) \delta u_{\cI}^{(j)} =0.
\end{equation}
Using the definition of $s^{(j)}$ in \eqref{eq:sjq} and assumption (iii) we conclude that $ \delta u_{\cI}^{(j)}=0$, which gives injectivity.
\end{proof}

\subsection{The complex Schr\"odinger problem}
\label{sec:cplxq}
We consider the Schr\"odinger problem with $\sigma \in (0,\infty)^{\cE}$ with Schr\"odinger potential of the form $q(\omega) = q' + \jmath \omega q''$, where $q',q'' \in \real^{\cI}$ and $\omega \in \real$ is a frequency. The Dirichlet problem with boundary condition $f(\omega) \in \complex^{\cB}$ is
\begin{equation}
    \label{eq:qdir}
    \begin{aligned}
    (L_\sigma u(\omega))_{\cI\cI} + q(\omega) \odot [u(\omega)]_{\cI} &=0,\\
    [u(\omega)]_{\cB} &= f(\omega).
    \end{aligned}
\end{equation}
We measure the \emph{dissipated power} $q' \odot |u_{\cI}|^2$. As in the complex conductivity case \cref{sec:cplxcond}, we expect that data measured at two frequencies $\omega_0,\omega_1$ to be necessary for local uniqueness. We assume $\omega_0 = 0$ and denote $u_0 = u(\omega_0)$ (which is real) and $u_1 = u(\omega_1)$ (which is complex). Considering $u_1$ and $\overline{u_1}$ as independent variables we can formulate this inverse problem as follows.
\begin{problem}
    \label{prob:cplxq}
    Find $\gamma = (q',q'')$ and $u^{(j)} = (u_0^{(j)}, u_1^{(j)}, \overline{u_1}^{(j)})$, for $j=1,\ldots,N$ in the non-linear system \eqref{eq:nonlin} where
    \begin{equation}
        \cM(\gamma,u) = \begin{bmatrix}
        q' \odot |[u_0]_{\cI}|^2\\
        q' \odot |[u_1]_{\cI}|^2
    \end{bmatrix},
    ~
        \cL(\gamma,u) = \begin{bmatrix}
            (L_\sigma u_0)_{\cI} + q' \odot [u_0]_{\cI}\\
            (L_\sigma u_1)_{\cI} + (q' + \jmath \omega_1 q'') \odot [u_1]_{\cI}\\
            (L_\sigma \overline{u_1})_{\cI} + (q' - \jmath \omega_1 q'') \odot [\overline{u_1}]_{\cI}\\
            [u_0]_{\cB}\\
            [u_1]_{\cB}\\
            [\overline{u_1}]_{\cB}\\
        \end{bmatrix},
    \end{equation}
    and 
    \begin{equation}
        b^{(j)} = \begin{bmatrix}
            0;
            0;
            0;
            f_0^{(j)};
            f_1^{(j)};
            \overline{f_1}^{(j)}
        \end{bmatrix}, 
    \end{equation}
    where a semicolon denotes vector concatenation.
    Here the $f_0^{(j)}$ (resp. $f_1^{(j)}$) are known real (resp. complex) Dirichlet boundary conditions and $\omega_1$ is known.
\end{problem}
\Cref{prob:cplxq} has $2|\cI| + 3N|\cV|$ unknowns and $2N|\cI| + 3N|\cV|$ equations so it is formally determined for $N=1$ and overdetermined for $N>1$. The following theorem gives a sufficient condition for local uniqueness of \cref{prob:cplxcond}.

\begin{theorem}
    \label{thm:cplxq}
    Assume that:
    \begin{enumerate}[i.] 
        \item The graph $\cG$, $q' \in \real^{\cI}$ and the boundary conditions $f^{(j)}$ are such that \cref{thm:realq} is satisfied. 
        \item The Dirichlet problems for both $q'$ and $q' + \jmath \omega q''$ are well posed. 
        \item The Schr\"odinger potential real part $q'$ does not vanish: $|q'| > 0$.
    \end{enumerate}
    Define for $j=1,\ldots,N$ the $|\cI| \times |\cI|$ complex matrices
    \begin{equation}
      \label{eq:bj}
      B^{(j)} = -\jmath \diag([\overline{u_1}^{(j)}]_{\cI} ) [(L_{\sigma})_{\cI\cI}+\diag(q(\omega_1))]^{-1} \diag([u_1^{(j)}]_{\cI}).
    \end{equation}
    If $\Re([B^{(1)}; \ldots; B^{(N)}])$ is full rank, then the linearization of the two frequency complex inverse Schr\"odinger \cref{prob:cplxq} is injective.
\end{theorem}
\begin{proof}
    The proof proceeds in a similar fashion as the proof for \cref{thm:cplxcond}. We linearize \cref{prob:cplxq}, set the linearization to zero and show that this can only happen if $\delta q'$, $\delta q''$, $\delta u_0^{(j)}$, $\delta u_1^{(j)}$, $\delta \overline{u_1}^{(j)}$ are all zero. We can use \cref{thm:cplxq} on the equations in the linearization corresponding to $\omega_0 = 0$ to deduce that $\delta q'$ and $\delta u_0^{(j)}$ vanish. The next step is to solve for $\delta u_1^{(j)}$ and $\delta \overline{u_1}^{(j)}$ in terms of $\delta q''$. To see this, the linearization of the equations at the interior nodes involving $u_1^{(j)}$ gives
    \begin{equation}
        (L_\sigma \delta u_1^{(j)})_{\cI} + \delta q' \odot [u_1^{(j)}]_{\cI} + \jmath \omega_1 \delta q'' \odot [u_1^{(j)}]_{\cI} + q(\omega_1) \odot [\delta u_1^{(j)}]_{\cI}= 0.
    \end{equation}
    Using that $\delta q' = 0$ and $[\delta u_1^{(j)}]_{\cB} = 0$ and doing the similar operation on the conjugate equations we get
    \begin{equation}
        \label{eq:deltau1jq}
        \begin{aligned}
        [\delta u_1^{(j)}]_{\cI} &= - \jmath \omega_1 ((L_\sigma)_{\cI\cI} + \diag(q(\omega_1)))^{-1} \diag([u_1^{(j)}]_{\cI}) \delta q'',~\text{and}\\
        [\delta \overline{u_1}^{(j)}]_{\cI} &=  \jmath \omega_1 ((L_\sigma)_{\cI\cI} + \diag(\overline{q(\omega_1)}))^{-1} \diag([\overline{u_1}^{(j)}]_{\cI}) \delta q''.
        \end{aligned}
    \end{equation}
    We now move to the equation linearizing the dissipated power measurements for $\omega_1$. Setting to zero we obtain:
    \begin{equation}
        \delta q' \odot |[u_1^{(j)}]_{\cI}|^2 + q' \odot [\overline{u_1}^{(j)}]_{\cI} \odot [\delta u_1^{(j)}]_{\cI} + q' \odot [u_1^{(j)}]_{\cI} \odot [\delta \overline{u_1}^{(j)}]_{\cI} = 0.
    \end{equation}
    Using that $\delta q' = 0$, $|q'|>0$, the expressions for $\delta u_1^{(j)}$ and $\delta \overline{u_1}^{(j)}$ in \eqref{eq:deltau1jq} we obtain equations for $\delta q''$:
    \begin{equation}
        \begin{aligned}
            -\jmath  \diag([\overline{u_1}^{(j)}]_{\cI}) ((L_\sigma)_{\cI\cI} + \diag(q(\omega_1)))^{-1} \diag([u_1^{(j)}]_{\cI}) \delta q''&\\
           +\jmath \diag([u_1^{(j)}]_{\cI}) ((L_\sigma)_{\cI\cI} + \diag(\overline{q(\omega_1)}))^{-1} \diag([\overline{u_1}^{(j)}]_{\cI}) \delta q'' &= 0.
        \end{aligned}
    \end{equation}
    Recognizing $B^{(j)}$, we see that  $ B^{(j)} \delta q'' + \overline{B}^{(j)} \delta q'' = 2 \Re (B^{(j)}) \delta q'' = 0$ for $j=1,\ldots,N$. By assumption, $\Re([B^{(1)}; \ldots; B^{(N)}])$ is full rank therefore we can conclude that $\delta q'' =0$. This proves the desired injectivity result.
\end{proof}

\section{Numerical reconstruction with Gauss-Newton method}
\label{sec:gn}
So far we have focussed on whether the linearization of a non-linear problem admits a unique solution. The Gauss-Newton method can be used to solve the non-linear problem via successive linearization. Here the goal is to solve the non-linear system of equations $R(x)=0$. Let use the notation $x^{(k)}$ for the iterates, $R_k = R(x^{(k)})$ and $DR_k = \partial_xR(x^{(k)})$. The Gauss-Newton iteration  is as follows (see e.g.  
\cite{Nocedal:2006:NO})
\begin{center}
\begin{minipage}{0.6\textwidth}
\begin{tabbing}
$x^{(0)}=$ given\\
for \= $k=0,1,2,\ldots$\\
\> $x^{(k+1)} = x^{(k)} - (DR_k^T DR_k + \alpha^2 I)^{-1} DR_k^T R_k$\\
end
\end{tabbing}
\end{minipage}
\end{center}
Here $\alpha$ is a regularization parameter. As an illustration, we implemented the Gauss-Newton iteration to find a real conductivity from dissipated power measurements (\cref{prob:dhic}) on a grid graph as appears in \cref{fig:gn}. Note that in practice the Gauss-Newton iteration needs a stopping criterion and a globalization strategy. For the stopping criterion we chose to check whether the gradient of $\frac{1}{2} \|R(x)\|_2^2$ is sufficiently small. For the globalization strategy we used the Armijo backtracking line search (see e.g. \cite{Nocedal:2006:NO}). More details are provided in \verb|gauss_newton.ipynb| in the supplementary materials \cite{thecode}.

The graph we considered was a $10 \times 10$ grid, with boundary nodes on the top, bottom, left and right edges. To define the boundary conditions, each node was given positions $x : \cV \to [0,1]$ and $y : \cV \to [0,1]$ so that they correspond to a uniform discretization of the unit square. With this in mind we defined the $N=2$ boundary conditions by $f^{(1)} = x_{\cB}+y_{\cB}$ and $f^{(2)} = x_{\cB}-y_{\cB}$. If the conductivity were constant, these boundary conditions are expected to give voltages that are (up to a constant) a finite difference discretization of the harmonic functions $x+y$ and $x-y$ on the unit square. This choice of boundary conditions avoids having edges with no current flowing. The conductivities reconstructed using the Gauss-Newton method are given in \cref{fig:gn}(c) without noise and in \cref{fig:gn}(d) with 5\% additive noise relative to the maximum value of the dissipated power in the network. In the noiseless case, the reconstructed conductivity has relative error of about $3.3\times 10^{-8} \%$ with respect to the true conductivity, while in the noisy case the relative error was of about $5.1\%$. For reference, we include in \cref{fig:gn}(a) and \cref{fig:gn}(b), the dissipated power used for the noiseless reconstructions.

\begin{figure}
\begin{center}
\begin{tabular}{c@{}c@{}c}
 (a) & (b) & \\
\includegraphics[height=4cm]{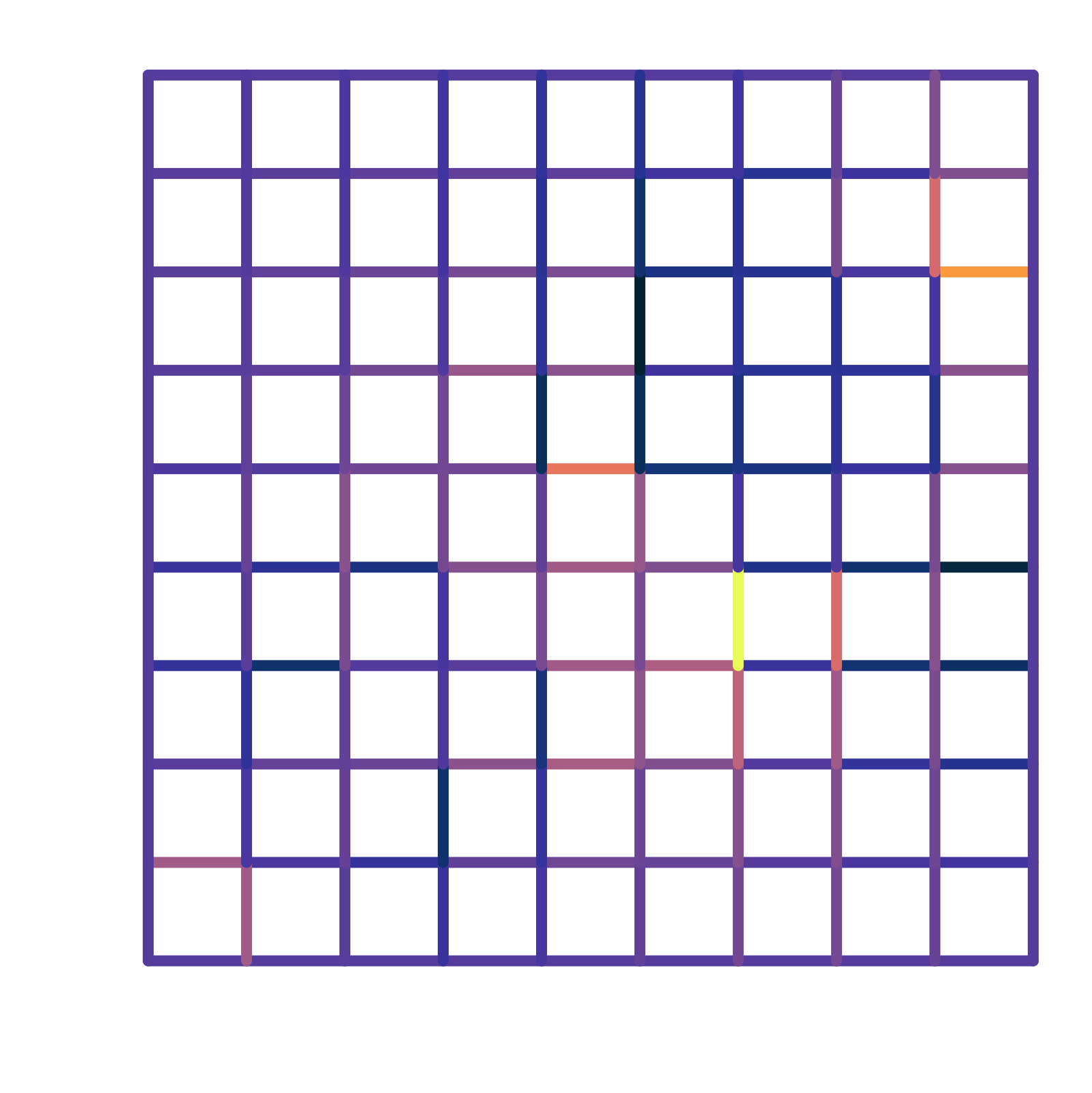} &
\includegraphics[height=4cm]{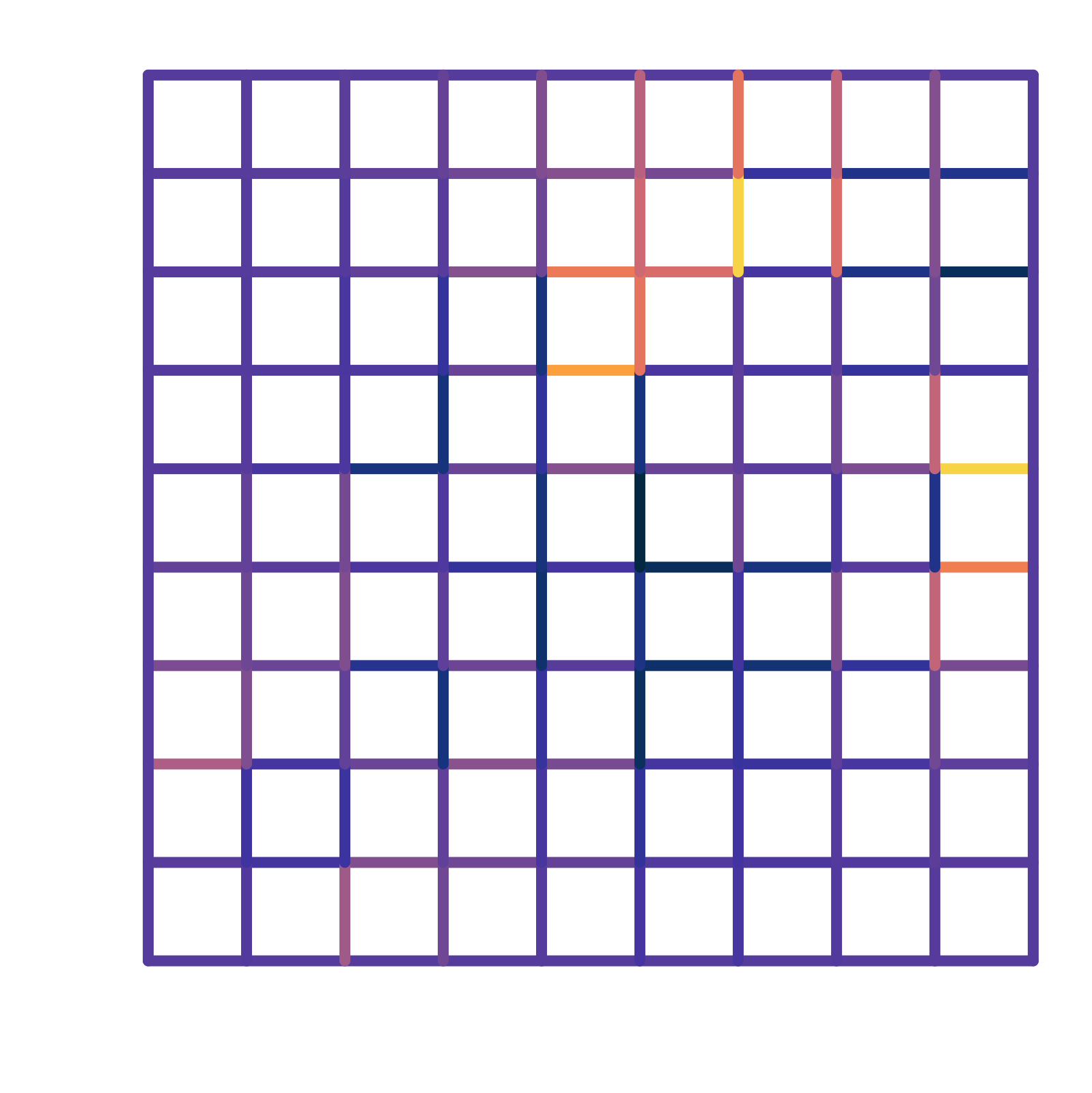} &
\includegraphics[height=4cm]{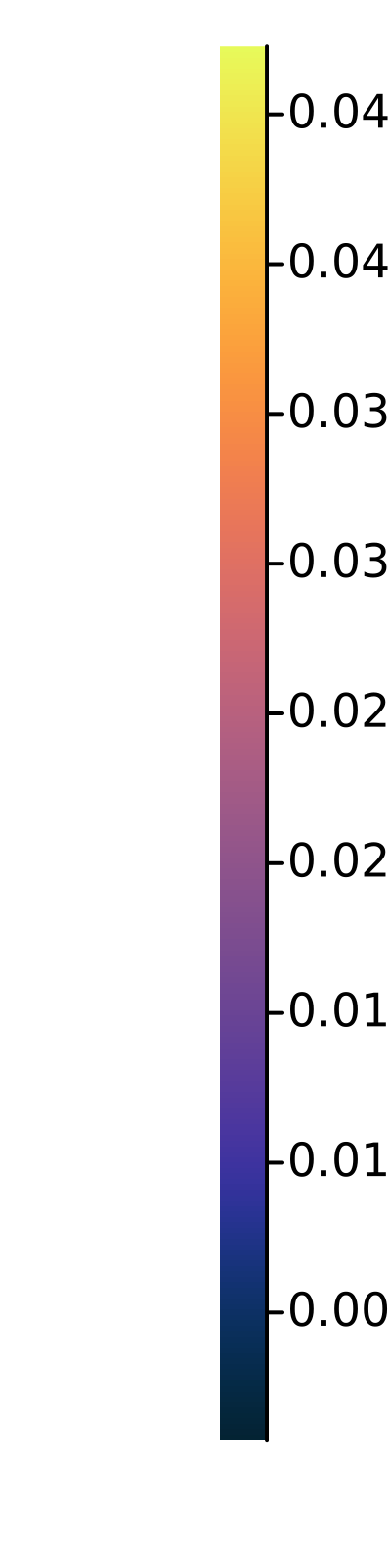}\\
\includegraphics[height=4cm]{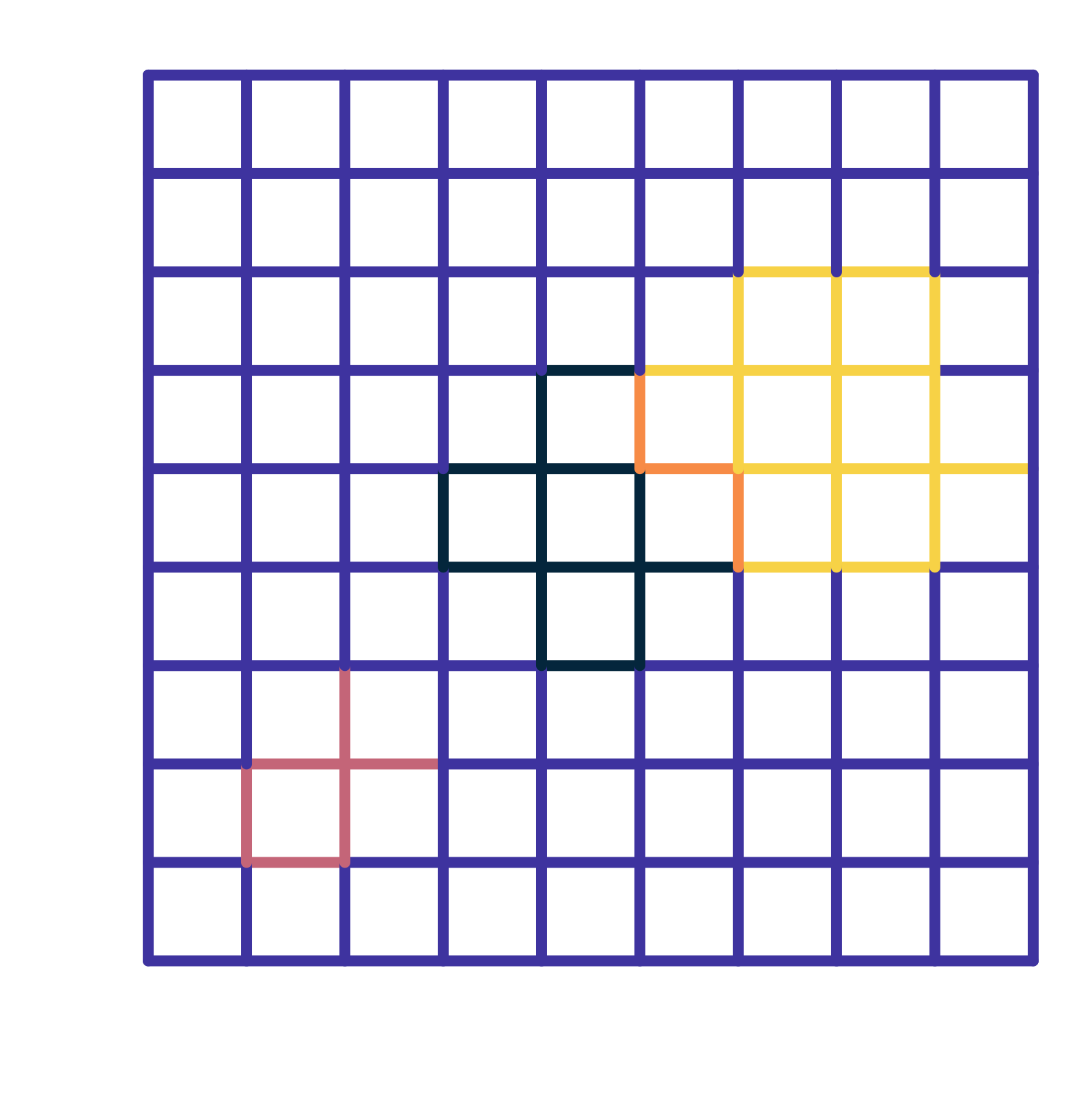} &
\includegraphics[height=4cm]{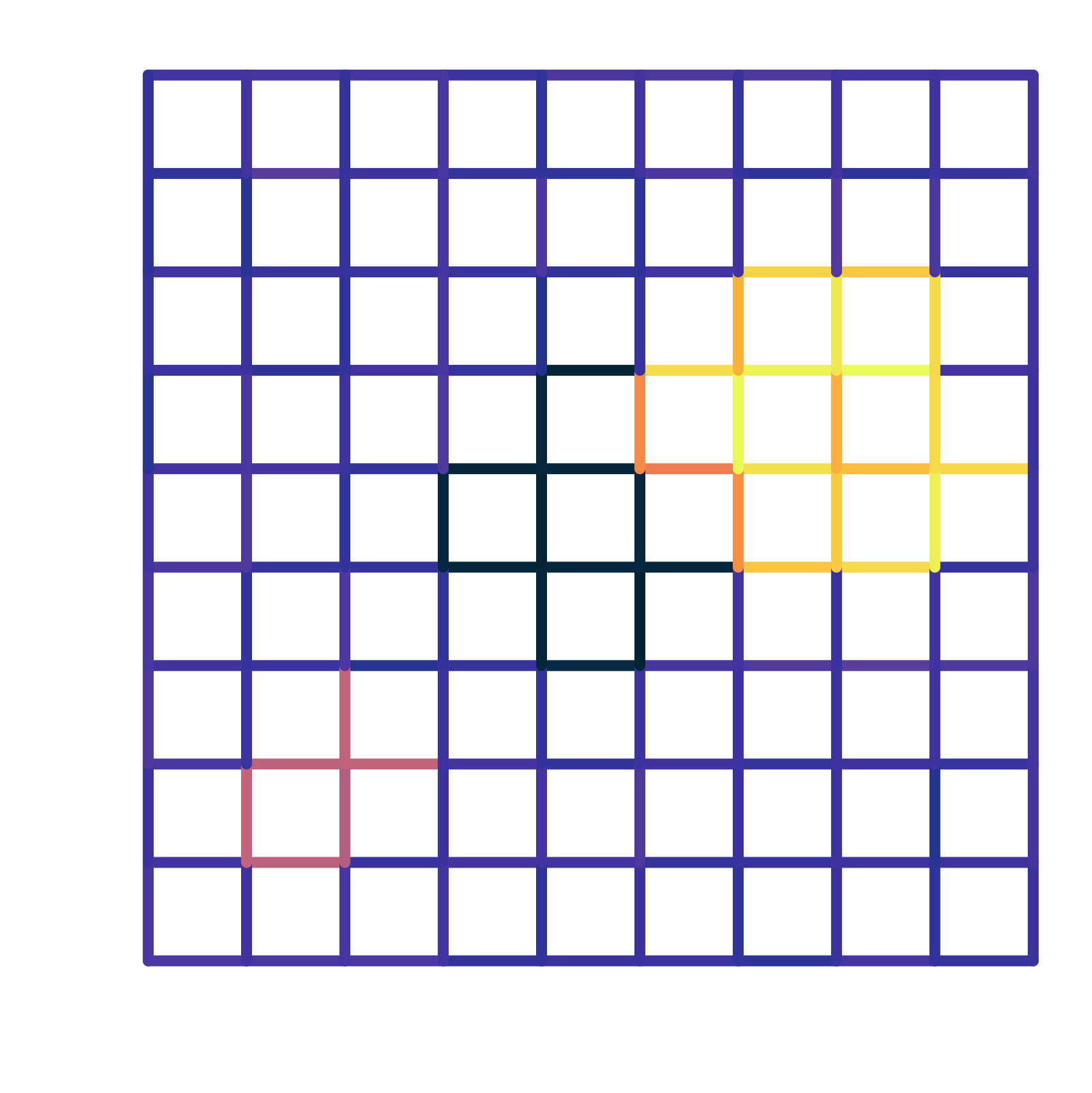} &
\includegraphics[height=4cm]{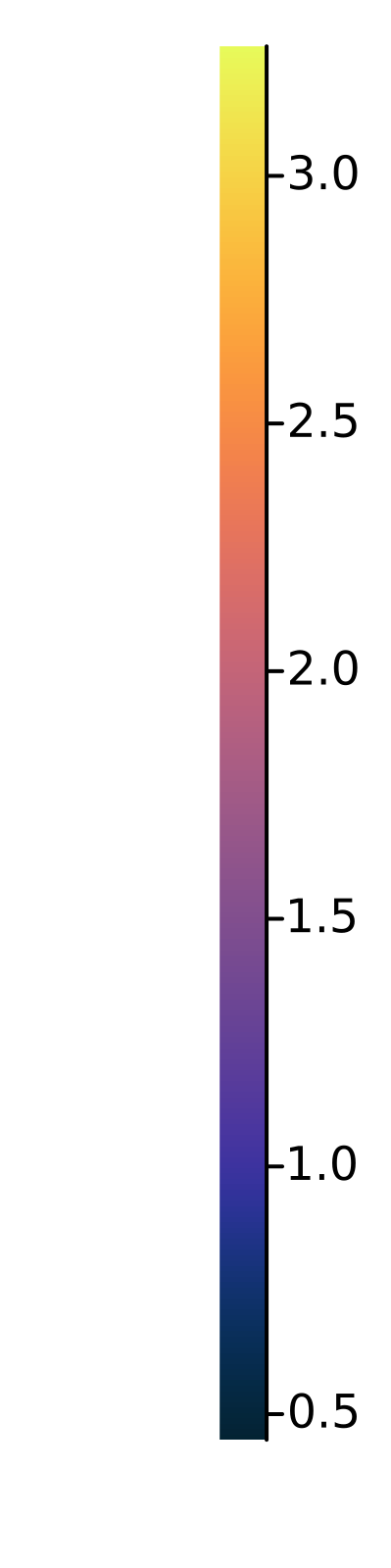}\\
(c) & (d) &
\end{tabular}
\end{center}
\caption{The Gauss-Newton method was used to find a real conductivity from dissipated power data (a) and (b). The reconstructions without noise are given in (c) and those with 5\% additive noise are given in (d).}
\label{fig:gn}
\end{figure}
\section{Summary and perspectives}
\label{sec:summary}
We have introduced a systematic approach to studying discrete inverse problems with internal functional data, that is inspired by the continuum approach \cite{Kuchment:2012:SIP,Bal:2014:HIP}. This approach has been applied to the inverse conductivity and Schr\"odinger problems from measurements of dissipated power at each of the network elements.  One possible extension of our results would be to measurements of \emph{complex power}, i.e. in the conductivity case $(\sigma' + \jmath \omega \sigma'') \odot |\nabla u(\omega)|^2$. The Schr\"odinger inverse problem may be extended as well to complex conductivity, but we assumed for simplicity $\sigma$ real. For the inverse conductivity problem, it would be interesting to apply the same technique for measurements of $|\sigma \odot \nabla u^{(j)}|$ \cite{Knox:2019:ENP} or $\sigma\odot \nabla u^{(j)}$ \cite{Ko:2017:RTI}. When the unknown (conductivity or Schr\"odinger potential) was allowed to be complex and dissipated power was measured, we relied on two frequency measurements $\omega_0$ and $\omega_1$ where $\omega_0=0$ to help us is in recovering the imaginary part of the unknown. It would be interesting to study the case where we have measurements at frequencies $\omega_1,\omega_2,\ldots,\omega_M$ that do not necessarily include the zero frequency. Also in the complex case we used the very simple case where the imaginary part of the unknown (conductivity or Schr\"odinger potential) is linear in the frequency $\omega$. This is a simplification: passive electrical elements can have a rational function response, as can be seen from network synthesis \cite{Bott:1949:ISW}. It would be interesting to study to what extent these more complicated network elements can be recovered from internal measurements of power.

\subsection*{Author Contributions} FGV proposed the problem, oversaw the research, contributed to the writing and numerical experiments. GY proved the uniqueness results for the real and complex conductivity problems, contributed to the writing, examples and numerical experiments. MC contributed with the linearization for the complex case. AR contributed with the proofs for the Schr\"odinger problem with real Schr\"odinger potential.

\subsection*{Acknowledgements} The authors acknowledge support by National Science Foundation Grants DMS-2008610, DMS-2136198 and from the University of Utah Mathematics Department. The research was first proposed by FGV to an introduction to research class in the Fall 2021, and followup REUs with GY, MC and AR. FGV would like to acknowledge the contributions of the other students that participated in the class: Brian Bettinson, Cormac LaPrete, and Calvin Zylstra.

\bibliographystyle{siamplain}
\bibliography{hip}
\end{document}